\documentclass[1pt]{article}
\usepackage{enumerate}
\usepackage{graphicx}
\usepackage{amsfonts,amssymb,amsmath,amsthm}
\usepackage[sort,authoryear,round]{natbib}
\usepackage[margin=1in]{geometry}
\usepackage{pdfsync}
\usepackage{hyperref}
\hypersetup{colorlinks=true,citecolor=blue,linkcolor=blue,urlcolor=blue}

\usepackage{hypernat}
\usepackage{color}

\usepackage{aliascnt}
\newtheorem{theorem}{Theorem}
\newaliascnt{lemma}{theorem}
\newtheorem{lemma}[lemma]{Lemma}
\aliascntresetthe{lemma}

\newaliascnt{cor}{theorem}
\newtheorem{corollary}[cor]{Corollary}
\aliascntresetthe{cor}

\newaliascnt{def}{theorem}
\newtheorem{definition}[def]{Definition}
\aliascntresetthe{def}

\newcommand{\given}{\mbox{ }\vert\mbox{ }}

\newcommand{\E}{\mathbb{E}}

\newcommand{\R}{\mathbb{R}}

\newcommand{\norm}[1]{\left\lVert#1\right\rVert}
\newcommand{\email}[1]{\href{mailto:#1}{#1}}

\newcommand{\nik}[2]{\mathcal{N}_p(#1,#2)}
\newcommand{\nikol}{\nik{\beta}{C}}
\newcommand{\nikola}[1]{\mathcal{N}_#1(\beta,C)}
\newcommand{\x}{\mathbf{x}}
\renewcommand{\u}{\mathbf{u}}
\renewcommand{\j}{\mathbf{j}}
\renewcommand{\hat}{\widehat}

\title{Minimax density estimation for growing dimension}
\author{Daniel J. McDonald\\Department of Statistics\\ Indiana University\\
  Bloomington, IN, 47405\\\email{dajmcdon@indiana.edu}}
\date{Version: \today}

\begin{document}

\maketitle

\begin{abstract}
 This paper presents minimax rates
 for density estimation when the data dimension $d$ is allowed to grow
 with the number of observations $n$ rather
 than remaining fixed as in previous analyses. We prove a
 non-asymptotic lower bound
 which gives the worst-case rate over standard classes of smooth densities,
 and we show that kernel density estimators achieve this rate. We also
 give oracle choices for the bandwidth and derive the fastest rate
 $d$ can grow with $n$ to maintain estimation consistency.
\end{abstract}

\section{Introduction}
\label{sec:introduction}

A convincing argument for the use of sparsity or other structural
priors in machine learning and statistics often begins with a
discussion of the ``curse of dimensionality''
\citep[e.g.][]{donoho2000high}. Unmistakable evidence 
of this curse is simply demonstrated in the fundamental scenario of
non-parametric density
estimation: the best estimator has squared $L^2$ error on the order of
$n^{-4/(4+d)}$ given $n$ independent observations in $d$
dimensions, a striking contrast with the parametric rate
$d/n$. If $d$ is even moderately large (but fixed), accurate
estimation requires significantly more data than if $d$ were small. In
fact, we will show that if $d$ is allowed to increase with $n$,
estimation accuracy
degrades even more quickly than the non-parametric rate above indicates. 

At first, it may seem that allowing $d$ to grow with $n$ is a rather
strange scenario, but the use of ``triangular array''
asymptotics is exceedingly common in the theory of high-dimensional
estimation. Theoretical results for the lasso, beginning at least with
\citep{GreenshteinRitov2004}, regularly adopt this framework allowing
the number of predictors to grow with
$n$. \citet{BuhlmannGeer2011} introduce the idea at the very beginning
of their foundational text, and it has been widely adopted in the
literature on regularized linear
models~\citep[e.g.][]{BelloniChernozhukov2011,BickelRitov2009,Meinshausen2007,YeZhang2010,NardiRinaldo2008a}. Under
this framework, the marginal 
distribution of the predictors has support whose dimension
is increasing with $n$. In the scenario of high-dimensional
regression, the dimension can increase very quickly (often on the
order of $d=o(n^\alpha),\ \alpha>1$) as long as most of these
dimensions are irrelevant for predicting the response. The extension
of these results for linear models to the non-linear scenario has been
studied mainly in the case of generalized (sparse) additive
models~\citep{RavikumarLafferty2009,RavikumarLiu2008,YuanZhou2015} which allow for
predictor specific non-linearities as long as the final predictions
are merely 
additive across dimensions. Fully nonparametric regression without the
additivity assumption has been completely ignored outside of the
fixed-$d$ framework, although it is a natural extension of the work
presented here.

Another motivation for appropriating the triangular array framework in
non-parametric density estimation is the burgeoning literature on
manifold
estimation~\citep{TalwalkarKumar2008,GenovesePerone-Pacifico2012,GenovesePerone-Pacifico2012a}. Given
high-dimensional data, a natural assumption is that the data is
supported on a low-dimensional manifold embedded in the
high-dimensional space. While estimating the manifold
is possible, we may also wish to estimate a density or a regression
function supported on the manifold. Recent work has focused on 
density estimation when the dimension of the manifold is fixed and
known~\citep{Asta2013,Hendriks1990,Pelletier2005,BhattacharyaDunson2010},
but the extension of such results to manifolds of growing dimension is
missing. Such an extension presumes
that the minimax framework we present can be extended to
manifolds. As pointed out by a reviewer, the short answer is yes. The
lower bound we derive applies immediately. The only
modification we need relates to our upper bound: the kernel should depend on the metric
given by the manifold rather than Euclidean distance as we use
here. 

A specific application of our setting would be from fMRI
data. Given a sequence of 3D resting-state fMRI scans from a
patient, researchers seek to estimate the dependence between
cubic centimeter voxels~\citep[e.g.][]{BullmoreSporns2009}. Each scan can contain on the
order of 30,000 voxels, while the number of scans for one individual
is smaller. It is too much to estimate the dependence between all
voxels, so the data are averaged into a small number ($\sim$20--200) of
regions. To estimate the dependence, standard methods assume
everything is multivariate Gaussian and estimate the covariance or
precision matrix. But the Gaussian assumption cannot be tested
without density estimates. Using our results, we could estimate
smooth densities. As the number of scans grows, we would want to
increase the number of regions. Our work illustrates how quickly the
number of regions can grow.

The remainder of this section introduces the statistical minimax
framework, discusses the specific data generating model we examine and
details notation,
presents some background on the estimator we use which achieves the
minimax rate, and gives a short overview of related literature. In
\autoref{sec:main-results}, we give our main results and discuss their
implications, specifically obtaining the fastest rate at which $d$ can
grow with $n$ to yield estimation
consistency. \autoref{sec:lower-bounds} gives the proof of our lower
bound over all possible estimators while the proof of the matching
upper bound for the kernel density estimator is given in
\autoref{sec:upper-bounds}. Finally, we discuss these results in
\autoref{sec:discussion}, provide some related results for other loss
functions, and suggest avenues for future research.

\subsection{The Minimax Framework}
\label{sec:minimax-framework}

In order to evaluate the feasibility of density estimation under the
triangular array, we use the statistical minimax framework. In our
situation, this
framework begins with a specific class of possible densities we are
willing to consider and
provides a lower bound on the performance of the best possible estimator
over this class. With this bound in hand, we have now quantified the
difficulty of the problem. If we can then find an estimator
which achieves this bound (possibly up to constants), then we can be
confident that this estimator performs nearly as well as possible for
the given class of densities. 
Thus, the minimax framework
reveals gaps between proposed estimators and the limits of possible
inference. Of course if the bounds fail to match, then we won't know
whether they are too loose, or the estimator is poor.

\subsection{Model and Notation}
\label{sec:model-notation}

We specify the following setting for density estimation in a
triangular array. Suppose for each $n\geq 1$, $X^{(n)}_i \in \R^{d(n)}$, $i=1,\ldots,n$
are independent with common density $f^{(n)}$ in some class which we
define below. For notational convenience, we
will generally suppress the dependence on $(n)$. To be clear,
in specifying this 
model, we do not assume a relationship for some sequence of densities
$\{f^{(n)}\}_{n=1}^\infty$, but rather we seek to understand the
limits of estimation when there is a correspondence between $d(n)$ and
$n$. Thus, we seek non-asymptotic results which characterize this
behavior. We will also employ the following notation: given
vectors $s,x\in\R^d$, let $|s|=\sum_i s_i$, $s! = \prod_i s_i!$ and
$x^s = x_1^{s_1}\cdots x_d^{s_d}$. Then define
\[
D^s = \frac{\partial^{|s|}}{\partial x_1^{s_1} \cdots \partial x_d^{s_d}}.
\]
Let $\lfloor \beta \rfloor$ denote the largest integer strictly less than
$\beta$. Throughout, we will use $a$ and $A$ for positive constants
whose values may change depending on the context.

Even were $d$ fixed at 1, it is clear that density estimation is
impossible were we to
allow $f$ to be arbitrary.\footnote{In the sense that, an adversary
  can choose a density and give us a finite amount of data on which
  our estimators will perform arbitrarily poorly.} For this reason, we will restrict the class
of densities we are willing to allow.
\begin{definition}
  [Nikol'skii class]
  \label{def:nikolskii}
  Let $p\in [2,\infty)$. The {\em isotropic Nikol'skii class} $\nikol$
  is the set of functions $f: \R^d \rightarrow \R$  such that:\\
  \ \ (i)\ $f \geq 0$ a.e.\\
  \ \ (ii)\  $\int f = 1$.\\
  \ \ (iii)\  partial derivatives $D^s f$
    exist whenever $|s|\leq\lfloor \beta \rfloor$\\
  \ \ (iv)\ 
    $
    \left[ \int \left( D^s f(x+t) - D^s f(x) \right)^p dx \right]^{1/p}
    \leq C \norm{t}_1^{\beta-|s|}, 
    $
    for all $t \in \R^d$.
\end{definition}
This definition essentially characterizes the smoothness of the
densities in a natural way. It can be shown easily that the Nikol'skii
class generalizes Sobolev and H\"older classes under similar
conditions~\citep[see e.g.][p.~13]{Tsybakov2009}.

\subsection{Parzen-Rosenblatt Kernel Estimator}
\label{sec:parz-rosenbl-kern}

Given a sample $X_1,\ldots,X_n$, the Parzen-Rosenblatt kernel density
estimator on $\R^d$ at a point $x$ is given by  
\[
\hat{f}_h(x) = \frac{1}{nh^d}\sum_{i=1}^n K\left(\frac{x-X_i}{h}\right).
\]
We will consider only certain functions $K$.
\begin{definition}\label{def:kernel}
  We say that $K: \R^d \rightarrow \R$ is an {\em isotropic kernel of order
    $\beta$} if $K(u) = G(u_1)G(u_2)\cdots G(u_d)$ for $G:
  \R\rightarrow \R$ satisfying
  $
  \int G = 1$, $\int |u|^\beta |G(u)|du<\infty$, and $\int u^j G(u) du = 0, 
  $
for  $0< j\leq\lfloor \beta\rfloor$.
\end{definition}
For the standard case $\beta=2$, the Epanechnikov kernel $G(u) =
0.75(1-u^2)I(|u|\leq 1)$ satisfies these conditions and is often the
default in software. The Gaussian kernel, $G(u)=(2\pi)^{-1/2}e^{-u^2/2}$, is also a member of this
class. For $\beta>2$, the kernel must take negative
values, possibly resulting in negative density
estimates, although, using the positive-part estimator eliminates this
pathology without affecting the results. Kernels for such $\beta$ can
be constructed using an orthonormal basis~\citep[see][p.~11]{Tsybakov2009}.

The intuition for this estimator is that it can be seen as a smooth
generalization of the histogram density estimator which uses local
information rather than fixed bins. Thus, if we believe the density is smooth, using
such a smoothed out version is natural. Another way to see this is to
observe that the kernel estimator is the
convolution of $K$ with the empirical density function $f_n$, defined
implicitly via $\int_{-\infty}^x f_n(y) dy = F_n(x)=\frac{1}{n}\sum
I(x_i \leq x)$. Using the empirical density itself is an unbiased
estimator of the true density (and it satisfies the central limit
theorem for fixed $d$), but by adding bias through the kernel, we may
be able to reduce variance, and achieve lower estimation risk for
densities which ``match'' the kernel in a certain way.

In this work, we have chosen, for simplicity, to use isotropic kernels
and the isotropic Nikol'skii class of densities. Basically, densities
$f\in\nikol$ have the same degree of smoothness in all directions. The
same is true of the kernels which satisfy
\autoref{def:kernel}. Allowing anisotropic smoothness is a natural
extension, although the notation becomes complicated very
quickly. For the anisotropic case under fixed-$d$ asymptotics, see for
example~\citet{GoldenshlugerLepski2011}.

\subsection{Related Work}
\label{sec:related-work}

Density estimation in the minimax framework is a well-studied problem
with many meaningful contributions over the last six decades, and we do
not pretend to give a complete overview 
here. Recent advances tend to build on one of four frameworks: (1) the support of $f$, (2) the
smoothness of $f$, (3) whether the loss is adapted to the nature of the
smoothness, and (4) whether the estimator can adapt to different
degrees of
smoothness. For a comprehensive overview of these and other concerns,
an excellent resource is \citet{Goldenshluger2014} which
presents results for adaptive estimators over classes of varying
smoothness when the loss is not necessarily adapted to the
smoothness. It also contextualizes and compares existing work. For previous results most similar to those we present in
terms of function classes and losses,
see~\citet{HasminskiiIbragimov1990}. Other important work is given
in~\citet{GoldenshlugerLepski2011,DevroyeGyorfi1985,Gerard-Kerkyacharian1996,JuditskyLambert-Lacroix2004}. 

Unlike in the density estimation setting, there are some related
results in the information theory literature
which endeavor to address the
limits of estimation under the triangular array. Essentially, this work
examines the estimation of the joint distribution of a $d$-block of random
variables observed in sequence from an ergodic process
supported on a finite set of points. \citet{MartonShields1994} show
that if $d$ grows like $\log n$, then these joint distributions can be
estimated consistently. An extension of these results to the case 
of a Markov random field embedded in a higher dimension is given by
\cite{Steif1997}. Our results are slightly slower than these (see
\autoref{cor:d-rate}), but estimating continuous densities
rather than finitely supported distributions is more difficult.

\section{Main results}
\label{sec:main-results}

Our main results give non-asymptotic rates for density estimation
under growing dimension. It generalizes existing results in that, had
$d$ been fixed, we recover the usual rate. Deriving the minimax rate
for density estimation requires two components: (1) finding the risk
of the best possible estimator for the hardest density in our class
and (2) exhibiting an estimator which achieves this risk. Our results
are only \emph{rate minimax} in that the upper and lower bounds match
in $d$ and $n$, but constants may be different.

We first present the lower bound. Our proof is given in \autoref{sec:lower-bounds}.
\begin{theorem}[Lower bound for density estimation]
  \label{thm:lowerBound}
  For any $d \in \mathbb{Z}^+$, $\beta>1$, $p \in [2,\infty)$, choose $n > n^*$ with 
  \begin{align*}
    &n^* =  64\norm{\Gamma_0}_2^{-2d}
    \left[\norm{\Gamma_0}_p^{-(d+1)(2\beta+d)}
      C^{4\beta+d}\left(\frac{\sigma}{\varphi(1/\sigma)} \right)^{d(d+\beta)}
      \right]^{1/\beta}.
  \end{align*}
  Then,
  \begin{align*}
    &\inf_{\hat{f}} \sup_{f \in \nik{\beta}{C}} \E_f
      \left[\left(\frac{n^\beta} {d^d}
      \right)^{\frac{1}{2\beta+d}} \norm{f-\hat{f}}_p \right]
    \geq
      c\left(\frac{1}{8}\right)\frac{C}{2}\frac{\kappa^{-\beta}}{8^{1/p}},
  \end{align*}
  for $c(v)$ a function only of $v$ and $\kappa:=\frac
  {\varphi(1/\sigma)} {\sigma\norm{\Gamma_0}_2^2}$. The infimum is
  over all estimators $\hat f$.
\end{theorem}
This result says that there exists a triangular array $\{f^{(n)}\}$ of
densities in $\nikol$ so that the best risk we can hope to achieve over all possible
estimators $\hat{f}$ is 
\[
\E_f \left[\norm{f-\hat{f}}_p \right] =O\left(\left(\frac{d^d}{n^\beta}\right)^{\frac{1}{2\beta+d}}\right).
\] 
The specific constant $\kappa$ as well as the minimum $n^*$ are properties of the proof technique, so
their forms are not really relevant (except that $\kappa$ is independent of $n$
and $d$). Specifically, $\varphi(u)=(2\pi)^{-1/2}e^{-u^2/2}$ is the
standard normal density, $\sigma>0$ is the standard deviation to be
chosen,  and $\Gamma_0$ is a small perturbation we make
explicit below. One could make other choices for the ``worst case''
density which result in different values. We also note that here $C$
is the same constant in each equation (and in the remainder of the
paper): it quantifies the smoothness of the class $\nikol$.

Our second result shows that, for an oracle choice of the bandwidth
$h$, kernel density estimators can achieve this rate. That is, for
\emph{any} density in $\nikol$, the risk of the kernel density
estimator is optimal. The proof is given in \autoref{sec:upper-bounds}.

\begin{theorem}[Upper bound for kernels]\label{thm:upperBound}
  Let $f\in\nikol$. Let $K(u)$ be an isotropic kernel of order
  $\ell=\lfloor \beta \rfloor$ which satisfies
  $
  \int K^2(u)du <\infty.
  $
  Take  $d\in \mathbb{Z}^+$, $p\in[2,\infty)$. Finally, take $h=A
  (d^2n)^{-1/(2\beta+d)}$ for some constant $A$. Then, for $n$ large enough,
  \begin{align*}
    \sup_{f \in \nikola{p}} \E_f \left[ \norm{\hat{f}_h(x) - f(x)}_p
    \right] = O\left(\left(\frac{d^d}{n^\beta}\right)^{\frac{1}{2\beta+d}}\right).
  \end{align*}
\end{theorem}

Our results so far have been finite sample bounds (which nonetheless
depend on $d$ and $n$). However, we also wish to know how quickly $d$
can increase so that the estimation risk can still go to zero
asymptotically (estimation consistency). Clearly, to have any hope that kernel
density estimators are consistent, $d$ must increase quite
slowly with $n$.

\begin{corollary}
\label{cor:d-rate}
  If
  $d = o\left(\frac{\beta \log n}{W(\beta \log n)}\right),$
  then
  \begin{align*}
    \sup_{f \in \nikol} \E_f \left[\norm{ \hat{f}_h(x) - f(x)}_p
    \right] &= o\left(1\right).
  \end{align*}
\end{corollary}

Here $W$ is the Lambert $W$ function, implicitly defined as the inverse
of $u \mapsto u\exp(u)$. For $n$ large, one can show using a series expansion
that $W(\log n) = \log\log 
n -\log\log\log n + o(1)$. So essentially, we require $d$ to grow
just slightly slower than $\log n$, the information theoretic rate for estimating
finite distributions with a sample from an ergodic process (see \autoref{sec:related-work}).

While we have stated both main theorems in terms of expectations,
analogous high-probability bounds can be derived similarly without extra effort.

\section{Lower bound for density estimation}
\label{sec:lower-bounds}

The technique we use for finding the lower bound is rather
standard. The idea is to convert the problem of density estimation
into one of hypothesis testing. This proceeds by first noting that the
probability that the error exceeds a constant is a lower bound for the
risk. We then
further reduce this lower bound by searching over only a finite class
rather than all possible densities. Finally, we ensure that there are
sufficiently many members in this class which are well-separated from
each other but difficult to distinguish from the true
density. Relative to previous techniques for minimax lower bounds for
density estimation, the main difference in our proof is that we must
choose different members of our finite class such that they have the
right dependence on $d$. Our construction will make use of the
Kullback-Leibler divergence.
\begin{definition}
  [KL divergence]
  The Kullback-Leibler divergence between two probability measures $P$
  and $P'$ is
  \[
  KL(P,P') = \begin{cases} \int dP \log\frac{dP}{dP'} & P\ll P'\\
    \infty &\textrm{else}.
  \end{cases}
  \]
\end{definition}
If both $P$ and $P'$ have Radon-Nikodym derivatives with respect to
the same dominating measure $\mu$, then we can replace distributions
with densities and integrate with respect to $\mu$. As long as the KL
divergence between the true density and the alternatives is small on average,
it will be difficult to discriminate between them. Therefore, the
probability of falsely rejecting the true density will be large. The following
lemma makes the process explicit.

\begin{lemma}[\citealt{Tsybakov2009}]\label{thm:tsybakov-KL-minimax}
  Let $\mathcal{L}: \R^+ \rightarrow \R^+$ which is monotone increasing with
  $\mathcal{L} (0)=0$ and $\mathcal{L} \not\equiv 0$, and let $A>0$ such that
  $\mathcal{L} (A) >0$. 
  \begin{enumerate}
  \item Choose elements $\theta_0,\theta_1,\ldots,\theta_M,$ $M \geq
    1$ in some class $\Theta$;
  \item Show that $\rho(\theta_j,\ \theta_k)\geq 2\tau > 0,$ $\forall 0 \leq j < k
    \leq M$ for some semi-distance $\rho$;
  \item Show that $P_{\theta_j} \ll P_{\theta_0}$, $\forall j=1,\ldots,M$ and 
    \[
    \frac{1}{M}\sum_{j=1}^M KL(P_{\theta_j},\ P_{\theta_0}) \leq \alpha \log M,
    \]
    with $0<\alpha < 1/8$.
  \end{enumerate}
  Then for $\psi=\tau/A$ we have
  \[
  \inf_{\hat{\theta}} \sup_{\theta\in\Theta} \E_\theta\left[
    \mathcal{L}(\psi^{-1} \rho(\hat\theta,\ \theta))\right] \geq c(\alpha) \mathcal{L} (A),
  \]
  where $\inf_{\hat\theta}$ denotes the infimum over all estimators
  and $c(\alpha)>0$ is a constant depending only on $\alpha$.
\end{lemma}

To use this result, we first choose a base density $f_0$ and $M$
alternative densities in $\nikol$. We then
show that these densities are sufficiently well-separated from each
other in the $L^p$-norm, $p\in [2,\infty)$, that is we take
$\rho(u,u')=\rho(u-u')=\norm{u-u'}_p$. Finally, we show that the 
KL-divergence between the alternatives and $f_0$ is uniformly small,
and therefore small on average. Our proof will use
$\mathcal{L}(u)=u$, though, as discussed following the proof, other
choices of monotone increasing functions (e.g.~$\mathcal{L}(u)=u^2$)
simply modify the conclusion but not the proof.

In order to get the ``right'' rate, we need to choose a base density
and a series of small perturbations to create a large collection of
alternatives. Getting the perturbations to be the right size and allow
sufficiently many of them is the main trick to derive tight bounds. In
our case, it is the choice $\Gamma(\mathbf{u})$ (described below) that
has this effect. The multiplicative dependence on $d$ turns out to be
the necessary deviation from existing lower bounds. Determining that
this is the appropriate modification is an exercise in trial-and-error, and
even this seemingly minor one is enough to compel a complete overhaul
of the proof.

\paragraph{The densities.}

Define $f_0(\x)=\frac{1}{\sigma^d}\prod_{i=1}^d \varphi(x_i/\sigma)$
where $\varphi(u)$ is the standard
Gaussian density.

Let $\Gamma_0: \R\rightarrow \R^+$ satisfy 
\begin{align*}
(i) \quad &|\Gamma^{(\ell)}_0(u) - \Gamma_0^{(\ell)}(u')| \leq
      |u-u'|^{\beta-\ell}/2,\\
&\quad\forall u,u',\ \ell\leq\lfloor \beta
      \rfloor,\\
(ii) \quad & \Gamma_0\in C^{\infty}(\R),\\
(iii)\quad &\Gamma_0(u) >0 \Leftrightarrow u \in(-1/2, 1/2).
\end{align*}
There exist many functions satisfying these conditions: 
e.g. $\Gamma_0(u) = a^{-1} \frac{d}{du} \exp(-1/(1-4u^2))I(|u| < 1/2)$ for some $a> 0$, 
since it is infinitely continuously
differentiable and $\norm{\Gamma^{(s)}_0}_\infty$ is decreasing in $s$. 

Define $\Gamma(\u) = dC \prod_{i=1}^d \Gamma_0(u_i)$, and for any
integer $m>0$, let
\[
\gamma_{m,\j}(\x) = m^{-\beta} \Gamma(m\x-\j),\ \
j\in\{1,\ldots,m\}^d. 
\]

Note that $\gamma_{m,\j}(\x)>0\Leftrightarrow \norm{\x}_\infty \leq
1$. Finally, take $f_\omega(\x) = f_0(\x) + \sum_\j \omega(\j)
\gamma_{m,\j}(\x)$ where for any $\j$, $\omega(\j) \in \{0,1\}$ so that
$\omega = \{\omega(\j)\}_{\j}$ is a binary vector in $\R^{(m-1)^d}$.

Now, we show that $f_0$, $f_\omega$ are densities in
$\nik{\beta}{C}$. For $f_0$, this is a density which is
infinitely differentiable, so we choose $\sigma>0$ such that
$\norm{f_0 ^{(s)} (\x)}_p \leq C/2$. We also have that for any $\j$, the
functions $\gamma_{m,\j}$ are non-zero only on non-intersecting
intervals of the form $(0,\ldots,\frac{j_i}{m} \pm
\frac{1}{2m},\ldots,0)$, so for any $|s| < \beta$,
\begin{align*}
\norm{\sum_\j \omega(\j)\left[\gamma^{(s)}_{m,\j}(\x+\mathbf{t}) -
  \gamma^{(s)}_{m,\j}(\x)\right]}_p 
  &\leq dCm^{-\beta+|s|}
    \sup_{|z|<t}\norm{\Gamma_0^{(s)}(x+z)
  - \Gamma^{(s)}_0(x)}^d_p\\
& \leq 2^{-d} d C m^{-\beta+|s|}\sup_{z\in[0,1]}|z|^{d(\beta-|s|)} < C/2,
\end{align*}
$\forall m>0$, so, $f_\omega$ is sufficiently smooth by the triangle
inequality.  As long as $f_\omega$ is a density, we will have
$f_\omega \in \nik{\beta}{C}$. First, $\int
\Gamma_0 = 0$, so $\int f_\omega=1$.  It remains to show that
$f_\omega \geq 0$. We have
\begin{equation}
  \norm{\sum_\j \omega(\j)\gamma_{m,\j}}_\infty
  \leq m^{-\beta} 
  \norm{\Gamma}_\infty \leq dC m^{-\beta}  
  \norm{\Gamma_0}^d_\infty.
  \label{eq:max-perturb}
\end{equation}
The smallest value taken by $f_0$ on the interval $[-1,1]$ where we
are adding perturbations is $\inf_{u\in [-1,1]}f_0(u) =
(\varphi(1/\sigma)/\sigma)^d$ . So, it is sufficient to
require~\eqref{eq:max-perturb} to be smaller. Therefore, we require 
\[
m> \left[ dC\left(\frac{\sigma \norm{\Gamma_0}_\infty}
    {\varphi(1/\sigma)} \right)^d\right]^{1/\beta}.
\]

\paragraph{Sufficient separation of alternatives.} 
We have for any $f_\omega$, $f_{\omega'}$,
\begin{align*}
  \norm{f_\omega-f_{\omega'}}_p 
  &= \norm{\sum_\j
    (\omega(\j)-\omega'(\j))\gamma_{m,\j}}_p
  = m^{-\beta-d/p} H^{1/p}(\omega,\omega')\norm{\Gamma}_p,
\end{align*}
where $H$ is the Hamming distance between binary vectors. We will use
some of the $f_\omega$ as our collection of $M$ alternatives. But we
need to know how many there are in the collection that are far enough
apart. The following theorem tells us about the size of such a
collection.
\begin{lemma}[Varshamov-Gilbert; \citealt{Tsybakov2009}] Let $m\geq 8$. Then there is a  subset
  $\mathcal{D}$ of densities $f_\omega$ such that for all
  $\omega,\omega' \in \mathcal{D}$, $H(\omega,\omega')\geq m^d /
  8$ and $|\mathcal{D}| \geq \exp\{m^d/8\}$. 
\end{lemma}

We now restrict our collection of densities to be only those
corresponding to the set $\mathcal{D}$. Then,
\begin{align*}
  m^{-\beta-d/p} H^{1/p}(\omega,\omega')\norm{\Gamma}_p
  &\geq
  m^{-\beta-d/p} \left(\frac{m^d}{8}\right)^{1/p} dC \norm{\Gamma_0}_p^d
  = 8^{-1/p}dm^{-\beta} C\norm{\Gamma_0}_p^d.
\end{align*}

\paragraph{Constant likelihood ratio.}

We have that for distributions $P_0$ with density $f_0$ and $P_\omega$
with density $f_\omega\in \mathcal{D}$,
\begin{align*}
  KL(P_\omega, P_0)
  &= n\int_{\R^d} d\x f_\omega(\x) \log\frac{f_\omega(\x)}{f_0(\x)}\\
  &=n \int_{\R^d} d\x \left( \frac{1}{\sigma^d}\prod_{i=1}^d
    \varphi(x_i/\sigma) + \sum_{\j}
    \omega(\j)\gamma_{m,\j}(\x)\right)\\
  &\quad\times
    \left[ \log \left( \frac{1}{\sigma^d}\prod_{i=1}^d
    \varphi(x_i/\sigma) + \sum_{\j} \omega(\j)\gamma_{m,\j}(\x)\right)
   - \log \left(\frac{1}{\sigma^d}\prod_{i=1}^d
    \varphi(x_i/\sigma)\right)\right]\\
  &\leq n\int_{\R^d} d\x \left( \frac{1}{\sigma^d}\prod_{i=1}^d
    \varphi(x_i/\sigma) + \sum_{\j} \omega(\j)\gamma_{m,\j}(\x)\right)
    \left[\frac{ \sum_{\j} \omega(\j)\gamma_{m,\j}(\x)}
     {\frac{1}{\sigma^d}\prod_{i=1}^d 
    \varphi(x_i/\sigma)} \right]\\
  &=\int_{[0,1]^d} d\x \frac{ \left(\sum_{\j}
    \omega(\j)\gamma_{m,\j}(\x)\right)^2}
    {\frac{1}{\sigma^d}\prod_{i=1}^d 
    \varphi(x_i/\sigma)}\\
  &\leq
    n\left(\frac{\sigma\norm{\Gamma_0}_2^2}{\varphi(1/\sigma)}\right)^d
    d^2 C^2 m^{-2\beta}
\end{align*}
Therefore, 
we must choose $m$ so that for $n$, $d$, large enough, 
\[
n\left(\frac{\sigma\norm{\Gamma_0}_2^2}{\varphi(1/\sigma)}\right)^d
    C^2 m^{-2\beta}
\leq \alpha\log |\mathcal{D}|
\]
with $0<\alpha<1/8$. This is equivalent to requiring
\begin{align*}
8\left(\frac{\sigma\norm{\Gamma_0}_2^2}{\varphi(1/\sigma)}\right)^d
    d^2C^2 m^{-2\beta-d} \leq \frac{1}{8n}
\end{align*}
which is equivalent to
\[ m \leq \left[\frac{1}{(8C)^2}  (d^2n)
\left(\frac {\sigma\norm{\Gamma_0}_2^2}{\varphi(1/\sigma)}
\right)^{d}\right]^{\frac{1}{2\beta+d}}.
\]

\paragraph{Completing the result.}

Combining the results of the previous two sections gives us the
following lower bound on density estimators in increasing dimensions.
\begin{proof}[Proof of \autoref{thm:lowerBound}]
  Choose an integer $m = \norm{\Gamma_0}_p^{(d+1)/\beta}\kappa^{-1}_d
  (d^2n)^{1/(2\beta+d)}$ where for convenience we define $\kappa_d := (64C^2)^{1/{2\beta+d}} 
  \left(\frac{\varphi(1/\sigma)}{\sigma\norm{\Gamma_0}_2^2}
  \right)^{d/(2\beta+d)} \xrightarrow{d\rightarrow\infty} \kappa = \frac{\varphi(1/\sigma)}{\sigma\norm{\Gamma_0}_2^2}$. Note that
  $\kappa_d < \kappa$ for all  $d$ so $\kappa_d^{-1}>\kappa^{-1}$. 
  Then, we have the following:
  \begin{enumerate}
  \item The functions $f_0, f_\omega$ are densities in $\nik{\beta}{C}$ as, for
    $n>n^*$, $m>\left[ dC\left(\frac{\sigma \norm{\Gamma_0}_\infty} 
        {\varphi(1/\sigma)} \right)^d\right]^{1/\beta}.$
  \item For all $f_\omega$, $f_{\omega'} \in \mathcal{D}$,
  \begin{align*}
  \norm{f_\omega-f_{\omega'}}_p 
    \geq 8^{-1/p}dm^{-\beta} C \norm{\Gamma_0}_p^d
    &= 8^{-1/p}d \left(\norm{\Gamma_0}_p^{(d+1)/\beta}\kappa^{-1}_d
  (d^2n)^{1/(2\beta+d)}\right)^{-\beta} C \norm{\Gamma_0}_p^d\\
    &= 2 (8^{-1/p}) C\norm{\Gamma_0}_p
      \kappa_d^{-\beta} d^{d/(2\beta+d)} n^{-\beta/(2\beta+d)}\\
    &\geq 2\frac{C}{2} 8^{-1/p}\kappa^{-\beta}d^{d/(2\beta+d)}
      n^{-\beta/(2\beta+d)}
    =:2A\psi_{nd},
  \end{align*}
  where 
  $A=\frac{C}{2} 8^{-1/p}\kappa^{-\beta}$ and
  $\psi_{nd} = (d^dn^{-\beta})^{1/(2\beta+d)}$.
\item $\frac{1}{M}\sum_{\omega \in \mathcal{D}} KL(P_\omega,P_0) \leq
  \alpha \log |\mathcal{D}|$ since $\norm{\Gamma_0}_p^{(d+1)/\beta} <
  1$ for all $d,\beta$ by construction of $\Gamma_0$. Therefore, 
  \[
  m \leq \left[\frac{1}{8C^2}  (d^2n)
\left(\frac {\sigma\norm{\Gamma_0}_2^2}{\varphi(1/\sigma)}
\right)^{d}\right]^{\frac{1}{2\beta+d}}.
  \]
\end{enumerate}
Therefore, all the conditions of \autoref{thm:tsybakov-KL-minimax} are
satisfied. 
\end{proof}

We note that \autoref{thm:tsybakov-KL-minimax} actually allows more
general lower bounds which are immediate consequences of those
presented here. In particular, we are free to choose $\rho$ to be
other distances than $L^p$-norms, and we may take powers of those
norms or apply other monotone-increasing functions $\mathcal{L}$. For
example, this gives the standard lower bound under the mean-squared
error. We will not pursue these generalities further here, however, as
finding matching upper bounds is often more difficult, requiring
specific constructions for each combination $\mathcal{L}$ and
$\rho$. Deriving lower bounds for $1\leq p<2$ is also of interest,
although this requires more complicated proof techniques. The case of
$p=\infty$ is actually a fairly straightforward extension, and we
discuss it briefly in \autoref{sec:discussion}.

\section{Upper bound for kernels}
\label{sec:upper-bounds}

To prove \autoref{thm:upperBound}, we first use the triangle
inequality to decompose the loss into a bias component and a variance 
component:
\begin{align*}
\E\left[\norm{\hat{f}_h-f}_p\right] 
&\leq
\E\left[\norm{\hat{f}_h-\E\hat{f}_h}_p\right] +
\norm{\E\hat{f}_h-f]}_p \\
  &=: \E\left[\left(\int |\sigma(x)|^p \right)^{1/p}\right]  + \left(\int|b(x)|^p\right)^{1/p}.
\end{align*}
We now give two lemmas which bound these components separately. For
the bias, we will need a well known preliminary result.

\begin{lemma}
  [Minkowski's integral inequality]
  \label{lem:minkowski}
  Let $(\Omega_1, \Sigma_1,\mu_1),\ (\Omega_2, \Sigma_2,\mu_2)$ be
  measure spaces, and let $g: \Omega_1\times\Omega_2 \rightarrow
  \R$. Then for $p\in[1,\infty]$
  \begin{align*}
  \left[\int_{\Omega_2} \left|\int_{\Omega_1}
    g(x_1,x_2)d\mu_1(x_1)\right|^p d\mu_2(x_2)\right]^{1/p}
  &\leq \int_{\Omega_1} \left[\int_{\Omega_2}
    \left|g(x_1,x_2)\right|^p d\mu_2(x_2)\right]^{1/p} d\mu_1(x_1),
  \end{align*}
  with appropriate modifications for $p=\infty$.
\end{lemma}

\begin{lemma}
  \label{lem:bias}
  Let $f \in \nikol$ for $p\in [1,\infty)$ and let $K$ be an isotropic Kernel of order $\ell=\lfloor
  \beta \rfloor$.
  Then for all $h>0$, $d\geq 1$, and $n\geq 1$,
  \[
  \int |b(x)|^p dx := \int \left|\E \hat{f}_h(x) - f(x)\right|^p dx = O\left( d^p h^{p\beta}\right).
  \]
\end{lemma}

For the bias, the proof technique depends on the smoothness of the
density $f$ as well as the smoothness of the kernel. It also holds for
any $p\in[1,\infty)$.

\begin{proof}
  By Taylor's theorem
  \begin{align*}
    f(x+uh) = f(x) + \sum_{|s|=1} u^sh D^s f(x) + \cdots +
    \frac{h^\ell}{(\ell-1)!} \sum_{|s|=\ell} u^s \int_0^1
    (1-\tau)^{\ell-1} D^s f(x+\tau u h) d\tau.
  \end{align*}
  Since the kernel is of order $\ell$, lower order polynomials in $u$
  are equal to 0, so 
  \begin{align*}
    |b(x)| 
    &= \left|\int du \Omega_\ell(u)
    \left[ \sum_{|s|=\ell}
      u^s \int_0^1 d\tau
      (1-\tau)^{\ell-1} D^s f(x+\tau u h) \right] \right|\\
    &= \left|\int du \Omega_\ell(u)\left[ \sum_{|s|=\ell} u^s \int_0^1 d\tau
      (1-\tau)^{\ell-1} \Delta(x,\tau) \right] \right|,
  \end{align*}
  where $\Delta(x,\tau) = D^s f(x+\tau u h) - D^s f(x)$ and $\Omega_\ell(u) =
  K(u) \frac{h^\ell}{(\ell-1)!}$.
  Now applying \autoref{lem:minkowski} twice, 
  \begin{align*}
    \int |b(x)|^p dx
    &\leq \int dx \left( \int du |\Omega_\ell(u)|\norm{u}_1^\ell
      \int_0^1 d\tau 
      (1-\tau)^{\ell-1} \left|\Delta(x,\tau) \right|
      \right)^p \\
    &\leq \left(\int du
        |\Omega_\ell(u)|\norm{u}_1^\ell
      \left[\int dx \left( \int_0^1 d\tau
      (1-\tau)^{\ell-1} \left|\Delta(x,\tau) \right|
      \right)^p \right]^{1/p} \right)^p\\
    &\leq \left(\int du
        |\Omega_\ell(u)|\norm{u}_1^\ell 
    \int_0^1 d\tau (1-\tau)^{\ell-1} \left( \int dx
      \Delta(x,\tau)^p \right)^{1/p}\right)^p.
  \end{align*}
  Because $f \in \nikol$, we have
  \[
  \left( \int dx
  \Delta(x,\tau)^p \right)^{1/p} \leq C (\tau h\norm{u}_1)^{\beta-\ell}.
  \]
  So,
  \begin{align*}
    \int |b(x)|^p
    &\leq \left( \int
      du|\Omega_\ell(u)|\norm{u}_1^\ell
      \left[
      \int_0^1 d\tau(1-\tau)^{\ell-1} C (\tau h \norm{u}_1)^{\beta-\ell}
      \right]  \right)^p\\
    &= \left( \int  du |K(u)|\frac{ C \norm{u}_1^\beta h^\beta}{(\ell-1)!} \left[
      \int_0^1 d\tau (1-\tau)^{\ell-1} \tau^{\beta-\ell}
      \right] \right)^p\\
    &= AC h^{p\beta}\left(\sum_{i=1}^d \int |u_i|^\beta |G(u_i)|du_i\right)^p
    = O(d^p h^{p\beta}).
  \end{align*}
\end{proof}

Next we find an upper bound on the variance component. This result
does not depend on the smoothness of the density, only on properties
of the kernel. It does however depend strongly on $p$. Finally, note
that the result is non-random, so we can ignore the outer expectation.

\begin{lemma}
  \label{lem:var}
  Let $K: \R^d\rightarrow\R$ be a function satisfying
  $
  \int K^2(u)du < \infty.
  $
  Then for any $h>0$, $n\geq 1$ and any probability density $f$, and $p\geq1$,
  \begin{align*}
  \int |\sigma(x)|^pdx &= \int 
  \left(\hat{f}_h(x) - \E \hat{f}_h(x)\right)^p dx=O\left(\left(\frac{1}{nh^d}\right)^{p/2}\right).
  \end{align*}
\end{lemma}
\begin{proof}
  The proof  is an easy generalization of Proposition 1 in
  \citep{Masaon2009} and is omitted. For intuition, we simply present
  the case of $p=2$.
  \begin{align*}
    \int |\sigma(x)|^2 dx 
    & \leq \frac{1}{nh^{2d}} \E\left[
      K^2\left(\frac{X_i-x}{h}\right) \right]\\
    & = \frac{1}{nh^{2d}} \int \left[ \int f(z)
      K^2\left(\frac{z-x}{h}\right)dz \right]
      dx\\
    & = \frac{1}{nh^{2d}} \int f(z) \left[ \int 
                            K^2\left(\frac{z-x}{h}\right)dx \right]
      dz\\
    &= \frac{1}{nh^{d}} \int K^2(u)du
    = O\left(1/nh^d\right).
  \end{align*}
\end{proof}

With these results in hand, we can now prove \autoref{thm:upperBound}. 

\begin{proof}[Proof of \autoref{thm:upperBound}]
  Applying \autoref{lem:bias} and \autoref{lem:var} gives
  \[
  \sup_{f\in\nikol} \E\left[\norm{\hat{f}_h-f}_p\right] = O\left(dh^\beta\right)
  + O\left(\left(\frac{1}{nh^d}\right)^{1/2}\right).
  \]
  Taking 
  $
  h = A (d^2n)^{-\frac{1}{2\beta+d}}
  $
  balances the terms and gives the result.
\end{proof}

\section{Discussion}
\label{sec:discussion}

In this paper we have developed the first results for density
estimation under triangular array asymptotics, where both the number of
observations $n$ and the ambient dimension $d$ are allowed to
increase. Our results generalize existing, fixed-$d$ minimax results,
in that, were $d$ fixed rather than increasing, we would recover
previously known minimax rates (both lower and upper bounds). Our
results also show that kernel density estimators are minimax
optimal, which should come as no surprise, since they are minimax
optimal for fixed $d$. 

The results presented in this paper say essentially that, for $n$
large enough 
there exist
constants $0<a<A<\infty$ independent of $d,n$ such that for $n$ large
enough, 
\begin{align*}
  a \left(\frac{d^d}{n^\beta}\right)^{\frac{1}{2\beta+d}} 
  \leq \inf_{\hat{f}}\sup_{f\in\nikol} \E\left[ \norm{\hat{f}-f}_p
    \right]
  &\leq \sup_{f\in\nikol} \E\left[ \norm{\hat{f}_h-f}_p
    \right]
  \leq A \left(\frac{d^d}{n^\beta}\right)^{\frac{1}{2\beta+d}},
\end{align*}
for $p\in[2,\infty)$ when $\hat{f}_h$ is the kernel density estimator
with oracle $h$. This result generalizes immediately to a result
for $\E\left[ \norm{\hat{f}-f}_p^p\right]$. With longer proofs, we can
generalize this result to $\E\left[ \norm{\hat{f}-f}_p^s\right]$ for
some $s\neq p$ and to the case $p\in[1,2)$. Another extension is to
the case $p=\infty$ which picks up a factor of $\log n$ in the
numerator of the rate.

With the same techniques used here, we could also give results for
nonparametric regression under triangular array asymptotics. Given
pairs $(y_i,x_i)$, kernel
regression $g(x)$ can be written in terms of densities as
$
  g(x) = \E[Y\given X=x] = \int y f(x,y) dy / f(x)
$
 for joint and marginal densities $f(x,y)$ and $f(x)$ respectively. So
 results for the Nadaraya-Watson kernel estimator
\begin{align*}
  \hat{g}_h(x) = \frac{\sum_{i=1}^n y_i K((x-x_i)/h)}{\sum_{i=1}^n K((x-x_i)/h)}
\end{align*}
can be obtained with similar proof techniques to those presented here.

A related extension would consider the problem of conditional density
estimation directly. Using a similar form,
\begin{align*}
  \hat{q}_h(x,y) = \frac{\sum_{i=1}^n K_1((y_i-y)/h) K_2((x-x_i)/h)}{\sum_{i=1}^n K_2((x-x_i)/h)}
\end{align*}
estimates the conditional density $q(Y|X)$. If $X\in \R^d$, this
estimator has been shown to converge at a rate of
$O(n^{-\beta/(2\beta+1+d)})$ under appropriate smoothness
assumptions~\citep[see, e.g.][]{HallRacine2004}.

Our results also suggest some open questions. Wavelet density
estimators and projection estimators are known to be rate-minimax for
$d$ fixed  in that upper bounds match those of kernels in $n$, though
constants may be larger or smaller. Whether these methods also match
for increasing $d$ remains to be seen (the class of densities examined
is usually slightly different). Histograms are also useful
density estimators, and for fixed $d$, they are minimax over Lipschitz
densities with a slower rate than that for kernels, again because the
class of allowable densities is different. Upper bounds under the
triangular array with a similar form to
those presented here were shown
in~\citep{McDonaldShalizi2011a,McDonaldShalizi2015}, but deriving
minimax lower 
bounds for this class remains an open problem.
Extending our results to the manifold setting (as mentioned in 
\S\ref{sec:introduction}) is the most obvious path toward fast rates
for large $d$ and is left as future work.

\subsubsection*{Acknowledgements}

This material is based upon work supported by the National Science
Foundation under Grant No.~DMS--1407439 and the Institute for New
Economic Thinking under Grant No.~INO14--00020. The author thanks the
anonymous referees and the program committee for the $20^{th}$
International Conference on Artificial Intelligence and Statistics for their insightful
comments and Cosma Shalizi for comments on an early draft.

\bibliography{AllReferences}
\end{document}